\documentclass[a4paper,11pt]{amsart}
\usepackage[matrix,arrow,tips,curve]{xy}
\usepackage[english]{babel}
\usepackage{amsmath}
\usepackage{amssymb}
\usepackage{tikz}
\usepackage{enumerate}
\usepackage{graphicx}
\usepackage[colorlinks=true,linkcolor=blue,citecolor=green]{hyperref}
\usetikzlibrary{trees}
\usetikzlibrary{arrows}
\usepackage[pangram]{blindtext}

\newtheorem{thm}{Theorem}[section]

\newtheorem*{thm*}{Theorem}

\newtheorem{teo}[thm]{Theorem} 
\newtheorem{lema}[thm]{Lemma}
\newtheorem{cor}[thm]{Corollary}
\newtheorem{prop}[thm]{Proposition}
\newtheorem{rem}[thm]{Remark}

\definecolor{wwwwww}{rgb}{0.4,0.4,0.4}

\newcommand{\km}{k_{\textrm{max}}}

\hypersetup{pdfpagemode=UseNone}
\hypersetup{pdfstartview=FitH}
		
\numberwithin{equation}{section}	
\begin{document}
\thanks{The first author is partially suported by Conselho Nacional de Desenvolvimento Científico e Tecnológico - CNPq, code 434253/2018-9. The second author is partially supported by Coordena\c{c}\~ao de Aperfei\c{c}oamento de Pessoal de N\'ivel Superior - Brasil (CAPES) Finance Code 001. The third author was partially supported by FAPERJ/Brazil [E-26/200.386/2023]) and
CNPq/Brazil [403344/2021-2 \& 308067/2023-1].}
\title[A lower bound for the first eigenvalue of a minimal hypersurface in the sphere
]{A lower bound for the first eigenvalue of a minimal hypersurface in the sphere
}


\author[Asun Jiménez, Carlos Tapia Chinchay, Detang Zhou.]{Asun Jiménez, Carlos Tapia Chinchay, Detang Zhou.}

\address{\sc Asun Jiménez \\
Instituto de Matem\'atica e Estat\'istica, Universidade Federal Fluminense, Campus Gragoat\'a,   S\~ao Domingos\\
24210-200 Niter\'oi, Rio de Janeiro\\ Brazil.}
\email{majimenezgrande@id.uff.br}
\address{\sc Carlos Tapia Chinchay\\
Instituto de Matem\'atica e Estat\'istica, Universidade Federal Fluminense, Campus Gragoat\'a, S\~ao Domingos\\
24210-200 Niter\'oi, Rio de Janeiro\\ Brazil.}
\email{carlostapia@id.uff.br}
\address{\sc  Detang Zhou\\
Instituto de Matem\'atica e Estat\'istica, Universidade Federal Fluminense, Campus Gragoat\'a,  S\~ao Domingos\\
24210-200 Niter\'oi, Rio de Janeiro\\ Brazil.}
\email{zhoud@id.uff.br}

\subjclass[2020]{MSC 53C42 \and MSC 53C44}
\keywords{minimal hypersurfaces in $\mathbb{S}^n$,  first eigenvalue of the laplacian,   Yau's conjecture}

\begin{abstract}
Let $\Sigma$ be a closed embedded minimal hypersurface in the unit sphere $\mathbb{S}^{m+1}$  and let $\Lambda=\max\limits_{\Sigma}|A|$ be the norm of  its second fundamental form. In this work we prove that the first eigenvalue of the Laplacian of $\Sigma$ satisfies
 $$\lambda_1(\Sigma)> \dfrac{m}{2}+\frac{m(m+1)}{32(12\Lambda+m+11)^2+8},$$ and  $\lambda_1(\Sigma)=m$, when $\Lambda\le\sqrt{m}$.
 In particular, this estimate improves the  one  obtained recently in \cite{duncan2023improved}. The proof of our main result is based on a Rayleigh quotient estimate for a harmonic extension of an eigenfunction of the Laplacian of $\Sigma$ in the spirit of \cite{choi1983first}.
\end{abstract}

\maketitle


\section{Introduction.}
Let $\mathbb{S}^{m+1}$ denote the unit sphere in $\mathbb{R}^{m+2}$.
and  $\Sigma$ a closed embedded hypersurface within $\mathbb{S}^{m+1}$. The eigenvalues of the Laplacian operator $\triangle$ on $\Sigma$ form a discrete set of non-negative real numbers. We denote by $\lambda_1(\Sigma)$ the first nonzero eigenvalue. It is well known that  $\Sigma$ is minimal if and only if all coordinate functions in $\mathbb{R}^{m+2}$ restricted on $\Sigma$ are eigenfunctions corresponding to eigenvalues $m$. This implies that $\lambda_1(\Sigma)\le m$.

On the other hand, it is  interesting and important to find the sharp lower bound of $\lambda_1$ for minimal hypersurfaces in  $\mathbb{S}^{m+1}$.   Yau \cite{Y82}    conjectured that $\lambda_1(\Sigma)=m$. 
 Choi and Wang  \cite{choi1983first} proved that $\lambda_1(\Sigma)\geq \frac{m}{2}$.  This estimate was later refined in \cite{barros2004estimates}  by Barros-Bessa who   gave the lower bound
\begin{align}\label{1}
\lambda_1(\Sigma) > \frac{m}{2}.
\end{align} 
Many progress has been made towards proving Yau's conjecture after Choi-Wang's paper (see for instance \cite{barros2004estimates}, \cite{duncan2023improved}, \cite{tang2014isoparametric}, \cite{soret}, and \cite{zhao2023first}). 
Despite an extensive literature relating to the study of $\lambda_1(\Sigma)$ under additional assumptions on $\Sigma$, \eqref{1} has remained the strongest \textit{explicit} lower bound that is known to hold for a general embedded minimal hypersurface in $\mathbb{S}^{n+1}$.
The new estimate that we obtain in this work depends on the geometry of $\Sigma$ as we   explain in detail next.

Given $x\in \Sigma$, let $|A|(x):=\left(\sum\limits_{i=1}^mk_i^2(x)\right)^{1/2},$ where $k_{1}(x),\ldots,k_{m}(x)$ are the principal curvatures of $\Sigma$ in $x$ with respect to $\nu(x)$. We call $|A|$ \textit{the norm of the second fundamental form} $A_{\nu}$ and we define $\Lambda:=\max\limits_{\Sigma}|A|$.  It is known (see \cite{simons1968minimal}) that in the case $\Lambda\le\sqrt{m}$, then $\lambda_1(\Sigma)=m$. Therefore  our contribution concerns only the case $\Lambda> \sqrt{m}$.
 

 

 Precisely, we prove
\begin{teo}\label{TP}
 Let $\Sigma$ be a closed embedded minimal hypersurface in the unit sphere $\mathbb{S}^{m+1}$ and let $\Lambda=\max\limits_{\Sigma}|A|$ be the norm of  its second fundamental form. Then, the first eigenvalue of the Laplacian of $\Sigma$ satisfies 
  \begin{equation}\label{EstPrinc2}
  \lambda_1(\Sigma)> \dfrac{m}{2}+\frac{m(m+1)}{32(12\Lambda+m+11)^2+8},
 \end{equation}  and  $\lambda_1(\Sigma)=m$, when $\Lambda\le\sqrt{m}$.
\end{teo}

\begin{rem}
A recent improvement to \eqref{1} is given by Duncan-Sire-Spruck in \cite{duncan2023improved}, where they proved that 
\[
\lambda_1\geq \frac{m}{2}+\frac{a(m)}{\Lambda^6+b(m)},
\]
for specific functions $a(m,\Lambda)$ and $b(m)$ (see \eqref{ayb}). By simple comparison of the order of growth it is easy to see that our estimate is bigger  when $\Lambda$ is big enough.  Indeed a computation at the end of this paper shows  that it is bigger for every $m$ and $\Lambda$.
\end{rem}
\begin{rem} Our estimate depends on the  norm of second fundamental form and is not sharp.  It would be interesting to find a lower bound which is bigger than $\frac{m}2$ and depends only on $m$. 

We can combine \eqref{EstPrinc2} with the Yang-Yau inequality \cite{YY80} is an area bound for embedded minimal surfaces in $\mathbb{S}^3$ in terms of their genus. This plays a crucial role in the compactness theory of Choi-Schoen in \cite{CS85}.  They find a constant $C(\chi)$ which is an upper bound for  the norm of the second fundamental form of any compact minimal embedded minimal surface in $\mathbb{S}^{3}$  with Euler number $\chi$.

 On the other hand,  Yau’s conjecture is true for embedded minimal surfaces in $S^3$ which are invariant under a finite
group of reflections (see \cite{soret}). And  Zhao \cite{zhao2023first} proved that there is a lower bound depending on the genus. We also have the following corollary from Theorem \ref{TP}. 
\end{rem}

\begin{cor}\label{corTP}
 Let $C(\chi)$ be the constant in Choi-Schoen's theorem. Then the first nonzero eigenvalue of the Laplacian of a compact embedded minimal surface $\Sigma$ in $\mathbb{S}^3$ with Euler number $\chi$  satisfies 
 $$
  \lambda_1(\Sigma)> 1+\frac{3}{16(12C(\chi)+13)^2+4}.
 $$
\end{cor}

The rest of the paper is divided into two sections.

In Section \ref{2}, we recall the Reilly's formula and reformulate a result in \cite{barros2004estimates}. We also give a lower bound for $\lambda_1(\Sigma)$ in terms of the Rayleigh quotient of its harmonic extension. Namely

\begin{equation*}
  \lambda_1(\Sigma)\geq \dfrac{m}{1+\sqrt{1-\frac{m+1}{Q}}},
\end{equation*}
where
\begin{equation*} 
  Q:=\left(\displaystyle\int_{\Omega}|\overline{\nabla} u|^2dV\right)\left(\displaystyle\int_{\Omega}u^2dV\right)^{-1}
\end{equation*}
and $u$ is a solution to

\begin{equation*} 
  \left\{
    \begin{array}{ll}
     \overline{\triangle}u=0 & \textrm{in}\, \Omega \\
      u=\varphi, & \textrm{in}\, \Sigma,
    \end{array}
  \right.
\end{equation*}
where the bars in  the expressions  above refer to operations in the ambient sphere.
Here, $\varphi$ is an eigenfunction of the Laplacian $\triangle$ associated to the eigenvalue $\lambda_1(\Sigma)$ and $\Omega$ is a component of $\mathbb{S}^{m+1}-\Sigma$ which is chosen appropriately. We call $u$ \textit{the harmonic extension of $\varphi$ to $\Omega$.} 
Then we prove Theorem \ref{TP} assuming the validity of an appropriate  estimate for $Q$.

In Section \ref{sQ} we make a quick review of the normal exponential map and  we prove the estimate
\begin{equation}\label{PI}
\displaystyle\int_{\Omega}u^2 dV\geq C_2(m,\Lambda)\int_{\Omega}|\overline{\nabla} u|^2dV, 
\end{equation}
  which may be considered as an inverse Poincaré type inequality, i.e. $C_2$ is an upper estimate for $Q$.   We finish the paper by comparing our estimate with the one in \cite{duncan2023improved}.

It is important to note that, as in \cite{duncan2023improved}, we also use Reilly's formula and an upper bound of the mean curvature of the parallel surfaces to $\Sigma$. However, we provide an upper bound of $|\overline{\nabla} u|_{L^2(\Omega)}^2$ that depends only on the geometry of $\Sigma$, via an elementary result on harmonic extensions.

We need to point out that our techniques can be improved and generalized in order to obtain an estimate for the first eigenvalue of a minimal surface embedded in an ambient of bounded sectional curvature. However, the sharp bound can only be achieved by proving that $Q=m+1$. Further  work will be part of the PhD thesis of the second author. In fact, this work was in preparation before we had access to \cite{duncan2023improved}.

\section{A first eigenvalue estimate via Rayleigh quotient}\label{2}
In this section we will review Reilly's formula and give a lower bound for the nonzero first eigenvalue  $\lambda_1(\Sigma)$ in terms of the Rayleigh quotient of the harmonic extension of the corresponding eigenfunction. As a consequence, we will prove our main result, Theorem \ref{TP}, assuming the inequality \eqref{PI} which will be explicitly stated in   Theorem \ref{DIP}. 

 From now on, $\Sigma$ will denote a closed embedded hypersurface of $\mathbb{S}^{m+1}$. It follows that $\Sigma$  divides the sphere into two components $\Omega_1$ and $\Omega_2$, where $\partial \Omega_1=\partial\Omega_2=\Sigma$ (see \cite{choi1983first}). Set $\nu$ as the unit normal of  $\Sigma$ pointing outward to $\Omega_1$ ($-\nu$ as the unit normal of $\Sigma$ and pointing outward to $\Omega_2$) and $A_{\nu}$ the second fundamental form of $\Sigma$ with respect to $\nu$.

Let $\varphi\in C^\infty(\Sigma)$. We can assume, without loss of generality, it satisfies the property $$\int_{\Sigma}\langle A_{\nu}\nabla \varphi,\nabla\varphi\rangle dS\geq 0$$ and denote $\Omega:=\Omega_1$. Otherwise we can choose $\Omega=\Omega_2.$ Let us denote all the functions of class $C^2$ that extend the function $\varphi$ over $\overline{\Omega}$ as $C^2_{\varphi}(\overline{\Omega})$. The following equation is known as Reilly's formula (see \cite{reilly1977applications}).

\begin{lema}\label{lem2.2}
  For all $u\in C^2_{\varphi}(\overline{\Omega})$ we have
\begin{equation}\label{reylli}
 \displaystyle\int_{\Omega}\left[(\overline{\triangle} u)^2-|\overline{\nabla}^2 u|^2-\textrm{Ric}_{\mathbb{S}^{m+1}(1)}(\overline{\nabla} u,\overline{\nabla} u)\right]dV=\displaystyle\int_{\Sigma}\left[\langle A_{\nu}\nabla \varphi,\nabla \varphi\rangle+2\frac{\partial u}{\partial \nu}\triangle \varphi+mH_{\Sigma}(u_{\nu})^2\right]dS.
\end{equation}
 where $\overline{\triangle} u, \overline{\nabla} u$ and $\overline{\nabla}^2 u$ denote the Laplacian, gradient and Hessian  of $u$ in $\Omega$, while $\triangle \varphi$ and $\nabla \varphi$ denote  the Laplacian and the gradient of $\varphi$ in $\Sigma$ with respect to the induced metric of $\Omega$. On the other hand, $\frac{\partial u}{\partial \nu}:=\langle\nabla u,\nu\rangle$ denotes the outward normal derivative of $u$ in $\Sigma$, $\textrm{Ric}_{\mathbb{S}^{m+1}}$ is the Ricci tensor of $\mathbb{S}^{m+1}$ and $H_{\Sigma}:=\textrm{trace}(A_{\nu})$ is the mean curvature of $\Sigma$.
\end{lema}
 The following corollary   follows from \eqref{reylli} by using that  $|\overline{\nabla}^2 u|^2\geq \frac{1}{m+1}(\overline{\triangle} u)^2$.
\begin{cor}\label{cor1}
Let $\Sigma$ be a closed embedded  minimal hypersurface in the unit sphere $\mathbb{S}^{m+1}$. For any $v\in C^2(\overline{\Omega})$ we assume $\Omega $ is chosen so that $\int_{\partial \Omega}\langle A_{\nu}(\nabla (v|_{\partial \Omega})),\nabla(v|_{\partial \Omega})\rangle dS\geq 0$. Then
\begin{equation}\label{reylly2}
 \displaystyle\int_{\Omega}\left[\frac{m}{m+1}(\overline{\triangle} v)^2-m|\overline{\nabla} v|^2\right]dV\displaystyle -2\int_{\partial \Omega}\frac{\partial v}{\partial \nu}\triangle (v|_{\partial \Omega})dS\ge 0.
\end{equation}

\end{cor}

The following result is proven by Barros and Bessa in \cite{barros2004estimates}. We include here a simpler proof. 
\begin{lema}Let $\Sigma$ be a closed embedded  minimal hypersurface in the unit sphere $\mathbb{S}^{m+1}$. Assume that $\varphi$ is a eigenfunction of $\triangle$ on $\Sigma$ corresponding to the eigenvalue $\lambda_1(\Sigma)$ and that $u$ is its harmonic extension to $\Omega$, i.e.
\begin{equation}\label{extarmo}
  \left\{
    \begin{array}{ll}
     \overline{\triangle}u=0 & \textrm{in}\, \Omega, \\
      u=\varphi & \textrm{in}\, \Sigma.
    \end{array}
  \right.
\end{equation}
Then for all $t\in \mathbb{R}$,
\begin{equation}\label{pol}
 (2\lambda_1(\Sigma)-m)Qt^2+2\lambda_1(\Sigma)t+\dfrac{m}{m+1}\geq 0,
\end{equation}
where $Q$ is defined by 
\begin{equation}\label{Qu}
  Q:=\frac{\displaystyle\int_{\Omega}|\overline{\nabla} u|^2dV}{\displaystyle\int_{\Omega}u^2dV}.
\end{equation}

\end{lema}
\begin{proof}

%
  For $t\neq 0$, let $g$ be the solution of the problem
$$\left\{
    \begin{array}{ll}
     \triangle g=u, & \textrm{in}\,\Omega, \\
      g=t\varphi, & \textrm{in}\,\Sigma ,
    \end{array}
  \right.
$$
where $u$ is a solution to (\ref{extarmo}). Then, from (\ref{reylly2}) applied to $g$, we have
\begin{equation}\label{reylly3}
 \displaystyle\int_{\Omega}\left(\frac{m}{m+1}u^2-m|\overline{\nabla} g|^2\right)dV\displaystyle +2t\lambda_1(\Sigma)\int_{\partial \Omega}\varphi\frac{\partial g}{\partial \nu} dS\ge 0.
\end{equation}

On the other hand, by  \eqref{extarmo} and Stokes' formula 
 \begin{equation}\label{green1}
      \displaystyle\int_{\Omega}\langle\overline{\nabla} g,\overline{\nabla} u\rangle dV=-\displaystyle\int_{\Omega}g\overline{\Delta} u\,dV+\displaystyle\int_{\Sigma}g\frac{\partial u}{\partial\nu} dS=t\displaystyle\int_{\Sigma}\varphi\frac{\partial u}{\partial\nu} dS=t\displaystyle\int_{\Omega}|\overline{\nabla} u|^2\,dV.
    \end{equation}
    Hence, by \eqref{green1},
$$\begin{array}{rcl}
      0\le\displaystyle\int_{\Omega}|\overline{\nabla} g-t\overline{\nabla} u|^2\,dV & = &  \displaystyle\int_{\Omega}\left(|\overline{\nabla} g|^2-2t\langle\overline{\nabla} g,\overline{\nabla} u\rangle +t^2|\overline{\nabla} u|^2\right)dV\\
       &=&\displaystyle\int_{\Omega}|\overline{\nabla} g|^2dV-t^2\displaystyle\int_{\Omega}|\overline{\nabla} u|^2 dV.
    \end{array}$$
Therefore
\begin{equation}\label{des}
  \displaystyle\int_{\Omega}|\overline{\nabla} g|^2dV\geq t^2\displaystyle\int_{\Omega}|\overline{\nabla} u|^2dV.
\end{equation}
Similarly
 \begin{equation}\label{green2}
\displaystyle\int_{\Sigma}\varphi\frac{\partial g}{\partial\nu} dS= \displaystyle\int_{\Omega}\langle\overline{\nabla} u,\overline{\nabla} g\rangle dV+\displaystyle\int_{\Omega}u\overline{\Delta} g\,dV=t\displaystyle\int_{\Omega}|\overline{\nabla} u|^2\,dV+\displaystyle\int_{\Omega}u^2\,dV.
    \end{equation}
From (\ref{reylly3}) and  (\ref{green2}) we have
$$
 \displaystyle\int_{\Omega}\left(\frac{m}{m+1}u^2-m|\overline{\nabla} g|^2\right)dV\displaystyle +2t^2\lambda_1(\Sigma)\displaystyle\int_{\Omega}|\overline{\nabla} u|^2\,dV+2t\lambda_1(\Sigma)\displaystyle\int_{\Omega}u^2\,dV\ge 0,
$$
and so from \eqref{des},
\begin{equation}\label{reilly4}(2\lambda_1(\Sigma)-m)t^2\displaystyle\int_{\Omega}|\overline{\nabla} u|^2dV+2\lambda_1(\Sigma)t\displaystyle\int_{\Omega}u^2dV+\dfrac{m}{m+1}\displaystyle\int_{\Omega}u^2dV \geq 0.
\end{equation}
Then (\ref{pol}) follows by dividing this last inequality by $\displaystyle\int_{\Omega}u^2dV$ and using the definition of $Q$ in \eqref{Qu}.
\end{proof}
Next we obtain a first estimate for $\lambda_1(\Sigma)$ which is a corollary of Barros-Bessa's theorem.
\begin{thm}Let $\Sigma$ be a closed embedded  minimal hypersurface in the unit sphere $\mathbb{S}^{m+1}$. Then
   \begin{equation}\label{estQ}
 \lambda_1(\Sigma)\geq \dfrac{m}{1+\sqrt{1-\frac{m+1}{Q}}}.
\end{equation}
\end{thm}
\begin{proof}
First note that, choosing any  $t$  in  \eqref{pol} such that  $2\lambda_1(\Sigma)t+\dfrac{m}{m+1}<0$,  then we trivially have $\lambda_1(\Sigma)> \frac{m}{2}$. 
On the other hand,   by  choosing $t=\frac{-\lambda_1(\Sigma)}{(m+1)(2\lambda_1(\Sigma)-m)}$ in \eqref{pol} we have that
\[
\begin{split}
\frac{Q\lambda^2_1(\Sigma)}{(m+1)^2(2\lambda_1(\Sigma)-m)}&\geq \frac{2\lambda^2_1(\Sigma)}{(m+1)(2\lambda_1(\Sigma)-m)}-\frac{m}{m+1}\\
&= \frac{1}{(m+1)(2\lambda_1(\Sigma)-m)}\left(2\lambda^2_1(\Sigma)-m(2\lambda_1(\Sigma)-m)\right)\\
&=\frac{1}{(m+1)(2\lambda_1(\Sigma)-m)}\left(\lambda^2_1(\Sigma)+(\lambda_1(\Sigma)-m)^2\right)\\&\geq \frac{\lambda^2_1(\Sigma)}{(m+1)(2\lambda_1(\Sigma)-m)}  .
\end{split}\]
 
Hence  $Q\geq m+1$ (see also \cite{barros2004estimates}) and so (\ref{estQ}) is well defined.  

Again from (\ref{pol}) we have
\begin{equation*}
 2\lambda_1(\Sigma)(Qt^2+t)\geq mQt^2-\dfrac{m}{m+1}.
\end{equation*} 
Then 
\begin{equation}\label{E0}
\lambda_1(\Sigma)\geq \frac{m}{2}\max\limits_{t(Qt+1)>0}\beta(t)
\end{equation}
where $\beta(t):=1-\frac{1}{Qt+1}-\frac{1}{(m+1)(Qt^2+t)}$. Note that $\beta'(t)=\frac{Q(m+1)t^2+2Qt+1}{(m+1)t^2(Qt+1)^2}$. We have that at the points where $t(Qt+1)>0$, $\beta'(t)=0$ if and only if
$$t=\frac{-2Q-\sqrt{4Q^2-4Q(m+1)}}{2Q(m+1)}=\dfrac{-1-\sqrt{1-\frac{m+1}{Q}}}{m+1}=:t_0.$$ Note that, in particular, $\beta$ has no critical points in the interval $(0,+\infty)$. It follows that

\begin{equation}\label{t0}
     \max\limits_{t(Qt+1)>0}\beta(t)=\beta(t_0)= \dfrac{Q(m+1)t_0^2-1}{(m+1)t_0(Qt_0+1)}=\frac{-2}{t_0(m+1)}=\frac{2}{1+\sqrt{1-\frac{m+1}{Q}}}.
\end{equation}
The above equality is valid from the fact that $\beta'(t_0)=0$, i.e. $Q(m+1)t_0^2+2Qt_0+1=0$.

Therefore, from (\ref{E0}) and (\ref{t0}) we have (\ref{estQ}).

\end{proof}

Next we prove  Theorem \ref{TP}   assuming Theorem \ref{DIP}.

\vspace{0.2cm}
\begin{flushleft}
\emph{Proof of Theorem \ref{TP}.} 
\end{flushleft}
 
It is well known that when $\Lambda\leq \sqrt{m}$, $\Sigma$ is either a great sphere $\mathbb{S}^n$ or a Clifford torus and so $\lambda_1(\Sigma)=m$.  Therefore, we can now assume that $\Lambda\geq\sqrt{m}$.
 Let $\varphi\in C^{\infty}(\Sigma)$ be the eigenfunction corresponding to the first nonzero eigenvalue $\lambda_1(\Sigma)$ and $u\in C^2_{\varphi}(\overline{\Omega})$ its harmonic extension to $\Omega$ as before. It follows from Theorem \ref{DIP} that $Q< 4(12\Lambda+m+11)^2+1.$  Therefore we have from (\ref{estQ})  that
\begin{equation}\label{l1previo}
  \lambda_1(\Sigma)\geq \dfrac{m}{1+\sqrt{1-(m+1)Q^{-1}}}=\dfrac{m}{2}+m\left(\dfrac{1}{1+\sqrt{1-(m+1)Q^{-1}}}-\dfrac{1}{2}\right).
\end{equation}
On the other hand, since for all $0\leq x<1$, $$\frac{x}{8}<\frac{1}{1+\sqrt{1-x}}-\frac{1}{2},$$    we can consider  $x=(m+1)Q^{-1}$ in (\ref{l1previo})  and deduce from  Theorem \ref{DIP} that
\begin{equation}\label{l1previo1}
  \lambda_1(\Sigma)>\frac{m}{2}+\frac{m(m+1)}{8}Q^{-1}>\frac{m}{2}+\frac{m(m+1)}{32(12\Lambda+m+11)^2+8}.
\end{equation}
The proof is then complete.

\begin{flushright}
$\square$
\end{flushright}

\section{Gradient estimate via an inverse Poincar\'e - type  inequality.}\label{sQ}
 The aim of this section is to prove an
  inverse Poincaré-type inequality in Theorem \ref{DIP}. To do that we  recall first some preliminary results concerning the normal exponential map.

In what follows $N\Sigma:=\{(x,v):\,x\in \Sigma,v\in T_x^{\bot}\Sigma\}$ will denote the normal bundle of $\Sigma$, $U\Sigma:=\{(x,v)\in N\Sigma:\,|v|=1\}$ will denote the normal unit bundle of $\Sigma$ and $\textrm{exp}^{\bot}:N\Sigma\to \mathbb{S}^{m+1}$ defined by $\textrm{exp}^{\bot}(x,v):=\exp_x(v)$ will denote the normal exponential map on $\Sigma$.  Such map is well defined in $\Sigma$ since $\Sigma$ is embedded with compact closure on the sphere $\mathbb{S}^{m+1}$ (see for instance \cite{gray2003tubes} ).  Let $\theta_{\Sigma}:N\Sigma\to\mathbb{R}$ denote the Jacobian determinant of the normal exponential map $\textrm{exp}^{\bot}$. On the other hand, let $\Phi_t:\Sigma\to \mathbb{S}^{m+1}$ be defined by $\Phi_t(x):=\textrm{exp}^{\bot}(x,t\nu(x))$ and $$\Sigma_t:=\Phi_t(\Sigma)=\{\textrm{exp}^{\bot}(x,t\nu(x)):x\in \Sigma\}.$$   From Proposition 2.1 in  \cite{duncan2023improved}, if we define $k_{\textrm{max}}:=\max\limits_{\Sigma,i}|k_i|$ and 
\begin{equation}\label{TSigma}
  T_{\Sigma}:=\arctan(k_{\textrm{max}}^{-1}),
\end{equation}
we have
$$T_{\Sigma}\leq \sup\{t>0|\,\Phi_t:\Sigma\to \Sigma_t\,\textrm{is a diffeomorphism}\}=:t_{\ast}.$$ Similarly, we have $$T_{\Sigma}\leq -\inf\{t<0|\,\Phi_t:\Sigma\to \Sigma_t\,\textrm{is a diffeomorphism}\}:=t^{\ast}.$$

Defining $\textrm{minfoc}(\Sigma):=\min\{t_{\ast},t^{\ast}\}$, we have that
$$T_{\Sigma}\leq \textrm{minfoc}(\Sigma).$$

%
%
%
 The following lemma correspond to Lemma 10.9 in \cite{gray2003tubes}.


\begin{lema}\label{vol}
For all   $0\leq t<\textrm{minfoc}(\Sigma)$
$$\dfrac{d}{dt}\ln\theta_{\Sigma}(x,t\nu(x))\leq \sum\limits_{i=1}^mk_i(x)=0,$$
where $k_{1}(x),\ldots,k_{m}(x)$ are the principal curvatures of $\Sigma$ in $x$ with respect to $\nu(x)$. In addition
\begin{equation}\label{tetauno}
  \theta(t,x):=\theta_{\Sigma}(x,t\nu(x))\leq \theta_{\Sigma}(x,0)=1.
\end{equation}
 
\end{lema}

 The proof of the following Lemma is straightforward from Lemma 8.1 in \cite{gray2003tubes} and the definition of $\theta_{\Sigma}$.
\begin{lema}
 For all $0<b<\textrm{minfoc}(\Sigma)$, the application $\Phi:[0,b]\times \Sigma\to \Phi([0,b]\times \Sigma)$ defined by $\Phi(t,x):=\Phi_t(x)$ is a diffeomorphism. In particular for any continuous function  $F$ over $\Phi([0,b]\times \Sigma)$
we have
   \begin{equation}\label{INT}
     \int_{\Phi([0,b]\times \Sigma)}F(y)dV=\int\limits_0^b\int\limits_{\Sigma}F(\Phi(t,x))\theta(t,x)dSdt,
   \end{equation}
where $\theta(t,x)=\theta_{\Sigma}(x,t\nu(x))$.
\end{lema}


%
%
%
%
\begin{lema}\label{CP}
  For each $0\leq t<T_{\Sigma}$ and $x\in \Sigma$  we have that
  $$\cos t-k_i(x)\,\sin \,t= \dfrac{ \sin \,(\theta_i(x)-t)}{\sin\,\theta_i(x)} \ge \dfrac{\sin \,(T_{\Sigma}-t)}{\sin\,T_{\Sigma}}>0 ,$$
  where $\cot \theta_i(x)=k_i(x)$.
\end{lema}

\begin{proof}
  By  the definition of $T_{\Sigma}$  in \eqref{TSigma}  we have  that  $k_i(x)\leq |k_i(x)|\leq k_{\textrm{max}}=\cot(T_{\Sigma}).$ Therefore
  $$\cos t-k_i(x)\,\textrm{sin}\,t\geq \cos t-\cot T_{\Sigma}\,\sin \,t=\dfrac{\sin \,(T_{\Sigma}-t)}{\sin \,T_{\Sigma}}>0,$$
where we have used that $0<T_{\Sigma}-t\leq T_{\Sigma}<\pi/2$.
\end{proof}

Given  $\varphi\in C^{\infty}(\Sigma)$, we define $\widetilde{\varphi}(\Phi_t(x)):=\varphi(x)$. The function $\widetilde{\varphi}$ is called \textit{the normal extension of} $\varphi$ and it is well defined in the set
$$\{y\in \mathbb{S}^{m+1}:y=\Phi_t(x),\,x\in \Sigma,\,|t|<T_{\Sigma}\}.$$

\begin{lema}\label{EG2}
 For each $0\leq t<T_{\Sigma}$ and $x\in \Sigma$ we have
$$|\nabla^{T}\widetilde{\varphi}|^2(\Phi_t(x))\leq \dfrac{\sin ^2\,T_{\Sigma}}{\sin ^2\,(T_{\Sigma}-t)}|\nabla\varphi|^2(x),$$ where $\widetilde{\varphi}$ is the normal extension of $\varphi$ into $\Omega$, $\nabla^{T}\widetilde{\varphi}(\Phi_t(x))$ denotes the gradient of $\widetilde{\varphi}|_{\Sigma_t}$ in $y=\Phi_t(x)$ and $\Phi_t(x)=\Phi(t,x)$.
\end{lema}

\begin{proof}
  For $x\in \Sigma$, let $\{E_i(t):=P_{t}e_i\}$ be an orthonormal basis of $T_{\Phi_t(x)}\Sigma_t$ where $\{e_i\}$ is an orthonormal basis of $T_x\Sigma$ such that $A_{\nu(x)}e_i=k_i(x)e_i$ and $P_{t}e_i$ the parallel transport of $e_i$ along the geodesic $\gamma_x(t):=\Phi_t(x)$ from $\gamma_x(0)=x$ to $\gamma_x(t)=\Phi_t(x)$. It follows that

  \begin{equation}\label{dFi}
\begin{array}{rcl}
  d(\Phi_t)_xe_i &= & P_{t}((\cos t)\,e_i+(\sin \,t)\,\nu'(x)e_i) \\
    &=&P_{t}((\cos t)\,e_i-(\sin \,t)\,A_{\nu(x)}e_i)\\
   &= &(\cos t-k_i(x)\sin \,t)E_i(t).
\end{array}
  \end{equation}

 We have from \eqref{dFi} that
   $$\begin{array}{rcl}
     |\nabla^{T}\widetilde{\varphi}|^2(\Phi_t(x)) =\sum\limits_{i}\langle\nabla^{T}\widetilde{\varphi}(\Phi_t(x)),E_i(t)\rangle^2  & = &  \sum\limits_{i}\langle\nabla^{T}\widetilde{\varphi}(\Phi_t(x)),\dfrac{d(\Phi_t)_xe_i}{\cos t-k_i(x)\sin \,t}\rangle^2 \\
     &&\\
     &=&\sum\limits_{i}\dfrac{\langle\nabla^{T}\widetilde{\varphi}(\Phi(t,x)),d\Phi_{(t,x)}(0,e_i)\rangle^2}{(\cos t-k_i(x)\sin \,t)^2}\\
     &&\\
     &=&\sum\limits_{i}\dfrac{(d\left(\widetilde{\varphi}|_{\Sigma_t}\circ\Phi\right)_{(t,x)}(0,e_i))^2}{(\cos t-k_i(x)\sin \,t)^2}\\
     &&\\
     &=& \sum\limits_{i}\dfrac{\langle(0,\nabla\varphi(x)),(0,e_i)\rangle^2}{(\cos t-k_i(x)\sin \,t)^2}.
     \end{array}$$
And so
 \begin{equation}\label{gradextnor}
   |\nabla^{T}\widetilde{\varphi}|^2(\Phi_t(x))=\sum\limits_{i}\dfrac{\langle\nabla\varphi(x),e_{i}\rangle^2}{(\cos t-k_i(x)\sin \,t)^2}.
 \end{equation}
     Using Lemma \ref{CP}  in formula  (\ref{gradextnor}), the proof of the lemma is concluded.
\end{proof}

 Next we are going to construct a transition function wich will be a key technical  tool in the proof of our main result. For any $a<b$ we define
\begin{equation}\label{fi}
  \psi_{a,b}(t):=1-g\left(\dfrac{t^2-a^2}{b^2-a^2}\right),
\end{equation}
where the function  $g:\mathbb{R}\to [0,1]$ is defined by $$g(t):=\dfrac{f(t)}{f(t)+f(1-t)},\qquad f(t):=\left\{
                                                 \begin{array}{ll}
                                                   e^{-1/t}, & t>0 \\
                                                   0, & t\leq 0
                                                 \end{array}
                                               \right. .
$$

It follows that
\begin{enumerate}
  \item $g(t)\geq 0$    $\forall  t\in\mathbb{R}.$
  \item $g(t)=0$ $\forall t\in(-\infty,0].$
  \item $\lim\limits_{t\to 0^+}g(t)=0$ and $\lim\limits_{t\to 1^-}g(t)=1$.
  \item The function  $g'(t)=\dfrac{e^{\frac{1}{t(1-t)}}(1-2t+2t^2)}{(e^{\frac{1}{t-t}}+e^{\frac{1}{t}})^2(t-1)^2t^2},$
is such that $0<g'(t)\leq 2$, for all $t\in(0,1)$ and $\lim\limits_{t\to 0^+}g'(t)=\lim\limits_{t\to 1^-}g'(t)=0.$
\end{enumerate}

\begin{figure*}[h]
\begin{center}
\includegraphics[height=4cm,width=6cm]{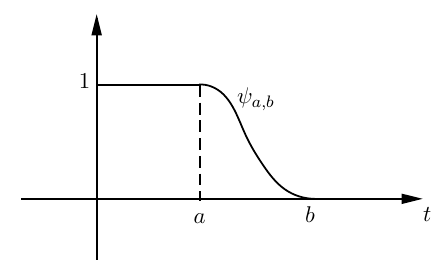}
\caption{Graph of $\psi_{a,b}$.}
\end{center}
\end{figure*}

For $ t\in[a,b],$ it follows  from \eqref{fi}  that
$$|\psi_{a,b}'(t)|=g'\left(\dfrac{t^2-a^2}{b^2-a^2}\right)\dfrac{2t}{b^2-a^2}\leq 2\left(\dfrac{2t}{b^2-a^2}\right)=\dfrac{4t}{b^2-a^2}.$$
Therefore
$$\displaystyle\int_{a}^{b}(\psi_{a,b}'(t))^2dt\leq\frac{16}{(b^2-a^2)^2}\displaystyle\int_{a}^{b}t^2dt=\dfrac{16(b^3-a^3)}{3(b^2-a^2)^2}.$$

 In particular, if we denote $\overline{\psi}_{\rho,c}:=\psi_{a,b}$ for the special choices $a=\frac{b}{c}$  for some $c>1$, and $b=\rho T_{\Sigma}$  for some $0<\rho<1$, we have that

\begin{equation}\label{Test}
  \displaystyle\int_{\frac{\rho T_{\Sigma}}{c}}^{\rho T_{\Sigma}}(\overline{\psi}_{\rho,c}'(t))^2dt\leq\dfrac{16}{3\rho T_{\Sigma}}\dfrac{c(c^3-1)}{(c^2-1)^2}.
\end{equation}

\begin{lema}\label{Gradv}
For  any  $0<\rho<1$ and $c>1$, the function $v_{\rho,c}:\overline{\Omega}\to \mathbb{R}$ defined by
\begin{equation}\label{vro}
  v_{\rho,c}(\Phi_t(x))= \overline{\psi}_{\rho,c}(t) \varphi(x)
\end{equation}
 satisfies
\begin{equation}\label{GradTest}
  \displaystyle\int_{\Omega}|\overline{\nabla} v_{\rho,c}|^2dV=\displaystyle\int_{0}^{\rho T_{\Sigma}}\int_{\Sigma}(\overline{\psi}_{\rho,c}'(t))^2\varphi^2(x)\theta(t,x)dS\,dt+\displaystyle\int_{0}^{\rho T_{\Sigma}}\int_{\Sigma}(\overline{\psi}_{\rho,c}(t))^2|\nabla^{T}\widetilde{\varphi}|^2(\Phi_t(x))\theta(t,x)dS\,dt.
\end{equation}
\end{lema}

\begin{proof}
Let $\overline{v}_{\rho,c}:[0,\rho T_{\Sigma}]\times \Sigma\to \mathbb{R}$ be the function defined by $\overline{v}_{\rho,c}(t,x):= \overline{\psi}_{\rho,c}(t)\varphi(x)$. It follows that $\overline{v}_{\rho,c}=v_{\rho,c}\circ \Phi$ and
\begin{equation}\label{Gravtilda}
  \overline{\nabla}\overline{v}_{\rho,c}(t,x)=( \overline{\psi}_{\rho,c}'(t)\varphi(x), \overline{\psi}_{\rho,c}(t)\nabla \varphi(x)).
\end{equation}

For $t\in[0,\rho T_{\Sigma}]$ and $x\in \Sigma$, consider $\{E_i(t)\}$ be as in the proof of Lemma \ref{EG2}. Then

$$\begin{array}{rcl}
   |\overline{\nabla} v_{\rho,c}|^2(\Phi_t(x))  & = & \sum\limits_{i}\langle\overline{\nabla} v_{\rho,c}(\Phi_t(x)),E_i(t)\rangle^2+\langle\overline{\nabla} v_{\rho,c}(\Phi_t(x)),\gamma_x'(t)\rangle^2 \\
     &  &  \\
     & = & \sum\limits_{i}\langle\overline{\nabla} v_{\rho,c}(\Phi(t,x)),\dfrac{d\Phi_{(t,x)}(0,e_i)}{\cos t-k_i(x)\sin \,t}\rangle^2+\langle\overline{\nabla} v_{\rho,c}(\Phi(t,x)),d\Phi_{(t,x)}(1,0)\rangle^2,
\end{array} $$
where the last equality is a consequence of (\ref{dFi}). It follows that

$$\begin{array}{rcl}
   |\overline{\nabla} v_{\rho,c}|^2(\Phi_t(x))  & = & \sum\limits_{i}\dfrac{\langle\overline{\nabla} v_{\rho,c}(\Phi(t,x)),d\Phi_{(t,x)}(0,e_i)}{(\cos t-k_i(x)\sin \,t)^2}\rangle^2+\langle\overline{\nabla} v_{\rho,c}(\Phi(t,x)),d\Phi_{(t,x)}(1,0)\rangle^2\\
&&\\
&=&\sum\limits_{i}\dfrac{\left(d(v_{\rho,c}\circ \Phi)_{(t,x)}(0,e_i)\right)^2}{(\cos t-k_i(x)\sin \,t)^2}+\left(d(v_{\rho,c}\circ \Phi)_{(t,x)}(1,0)\right)^2\\
&&\\
&=&\sum\limits_{i}\dfrac{\langle\overline{\nabla}\overline{v}_{\rho,c}(t,x),(0,e_i)\rangle^2}{(\cos t-k_i(x)\sin \,t)^2}+\langle \overline{\nabla}\overline{v}_{\rho,c}(t,x),(1,0)\rangle^2.
\end{array}$$
From (\ref{Gravtilda}) we have
\begin{equation}\label{gradprevio}
  \begin{array}{rcl}
 |\overline{\nabla} v_{\rho,c}|^2(\Phi_t(x))&=&\sum\limits_{i}\dfrac{\left(\overline{\psi}_{\rho,c}(t)\langle\nabla\varphi(x),e_i\rangle\right)^2}{(\cos t-k_i(x)\sin \,t)^2}+\left(\overline{\psi}_{\rho,c}'(t)\varphi(x)\right)^2\\
&&\\
&=&(\overline{\psi}_{\rho,c}(t))^2\sum\limits_{i}\dfrac{\langle\nabla\varphi(x),e_i\rangle^2}{(\cos t-k_i(x)\sin \,t)^2}+(\overline{\psi}_{\rho,c}'(t))^2\varphi^2(x)\\
&&\\
&=&(\overline{\psi}_{\rho,c}(t))^2|\nabla^{T}\widetilde{\varphi}|^2(\Phi_t(x))+(\overline{\psi}_{\rho,c}'(t))^2\varphi^2(x)
  \end{array}
\end{equation}

where the last equality is a consequence of (\ref{gradextnor}). On the other hand, from (\ref{INT}) we have
\begin{equation}\label{EE}
 \displaystyle\int_{\Omega}|\overline{\nabla} v_{\rho,c}|^2dV=\displaystyle\int_{0}^{\rho T_{\Sigma}}\int_{\Sigma}|\overline{\nabla} v_{\rho,c}|^2(\Phi_t(x))\theta(t,x)\,dS\,dt.
\end{equation}

  We conclude the proof of the lemma by replacing (\ref{gradprevio}) into (\ref{EE}).
\end{proof}

In order to  get  an upper estimate for $|\overline{\nabla} u|_{L^2(\Omega)}^2$, we will use the fact that the harmonic extension $u$ minimizes the Dirichlet energy in $C^{\infty}_{\varphi}(\Sigma)$.

\begin{lema}\label{EG1}
Let $u$ be the harmonic extension of $\varphi\in C^{\infty}(\Sigma)$. For all $v\in C^\infty_{\varphi}(\overline{\Omega})$ we have
  $$\displaystyle\int_{\Omega}|\overline{\nabla} u|^2dV\leq\displaystyle\int_{\Omega}|\overline{\nabla} v|^2dV.$$
\end{lema}

\begin{proof}
Since $u$ is harmonic, by Stokes' theorem   
  $$\displaystyle\int_{\Omega}\langle\overline{\nabla} v,\overline{\nabla} u\rangle dV=\int_{\Sigma}v\frac{\partial u}{\partial\nu}dS=\int_{\Sigma}u\frac{\partial u}{\partial\nu}dS=\int_{\Omega}|\overline{\nabla} u|^2dV.$$ And so
  $$\begin{array}{rcl}
    0 & \leq & \displaystyle\int_{\Omega}|\overline{\nabla} (u-v)|^2dV \\
     &  &  \\
     & = & \displaystyle\int_{\Omega}|\overline{\nabla} u|^2dV+\int_{\Omega}|\overline{\nabla} v|^2dV-2\int_{\Omega}\langle \overline{\nabla} u,\overline{\nabla} v\rangle dV \\
     &  &  \\
     & = &\displaystyle\int_{\Omega}|\overline{\nabla} u|^2dV+\int_{\Omega}|\overline{\nabla} v|^2dV-2\int_{\Omega}|\overline{\nabla} u|^2dV,
  \end{array}$$ and so we are done.
 
\end{proof}

 \begin{prop}\label{EG}
Let $u$ be the harmonic extension of $\varphi\in C^{\infty}(\Sigma)$. Then, if $\varphi$ is a first  eigenfunction of the laplacian satisfying $\displaystyle\int\limits_{\Sigma}\varphi^2\,dS=1$, we have

\begin{equation}\label{EG3}
  \displaystyle\int_{\Omega}|\overline{\nabla} u|^2\leq C_1,
\end{equation}
where
$C_1=C_1(\km):=\dfrac{32}{3\arctan(1/\km)}+\dfrac{\lambda_1(\Sigma)}{\sqrt{1+\km^2}}$.
 \end{prop}

\begin{proof}
  Let $0<\rho<1$, $c>1$ and let $v_{\rho,c}:\overline{\Omega}\to \mathbb{R}$ be the function defined by \eqref{vro}.

From Lemma \ref{Gradv} and Lemma \ref{EG1}    we have
\begin{equation}\label{estgrad1}
\begin{array}{rcl}
 \displaystyle\int_{\Omega}|\overline{\nabla} u|^2dV&\leq& \displaystyle\int_{\Omega}|\overline{\nabla} v_{\rho, c}|^2dV\\
  & = &  \displaystyle\int_{0}^{\rho T_{\Sigma}}\int_{\Sigma}(\overline{\psi}_{\rho,c}'(t))^2\widetilde{\varphi}^2(\Phi_t(x))\theta(t,x)dS\,dt+ \displaystyle\int_{0}^{\rho T_{\Sigma}}\int_{\Sigma}(\overline{\psi}_{\rho,c}(t))^2|\nabla^{T}\widetilde{\varphi}|^2(\Phi_t(x))\theta(t,x)dS\,dt\\
       &  &  \\
     &\leq&\displaystyle \int_{\frac{\rho T_{\Sigma}}{c}}^{\rho T_{\Sigma}}(\overline{\psi}_{\rho,c}'(t))^2\,dt+\displaystyle\int_{0}^{\rho T_{\Sigma}}\int_{\Sigma}(\overline{\psi}_{\rho,c}(t))^2|\nabla^{T}\widetilde{\varphi}|^2(\Phi_t(x))\,dS\,dt,
    \end{array}
\end{equation}

where the last inequality is a consequence of  the condition $\displaystyle\int\limits_{\Sigma}\varphi^2\,d\Sigma=1$ and \eqref{tetauno}.

 Then, by (\ref{Test}) and Lemma \ref{EG2} we can rewrite \eqref{estgrad1} as
\begin{equation}\label{intgradu}
 \begin{array}{rcl}
    \displaystyle\int_{\Omega}|\overline{\nabla} u|^2dV
     & \leq & \dfrac{16}{3\rho T_{\Sigma}}\dfrac{c(c^3-1)}{(c^2-1)^2}+\sin ^2(T_{\Sigma})\displaystyle\int_{0}^{\rho T_{\Sigma}}\csc^2(T_{\Sigma}-t)dt\displaystyle\int_{\Sigma}|\nabla\varphi|^2dS.
\end{array}
\end{equation}

 On the other hand, note that  $\lambda_1(\Sigma)=\displaystyle\int_{\Sigma}|\nabla\varphi|^2dS$.  Then it follows that
$$ \begin{array}{rcl}
  \sin ^2(T_{\Sigma})\displaystyle\int_{0}^{\rho T_{\Sigma}}\csc^2(T_{\Sigma}-t)dt\displaystyle\int_{\Sigma}|\nabla\varphi|^2dS
 &\leq&\lambda_1(\Sigma)\,\sin ^2(T_{\Sigma})\left(\cot((1-\rho)T_{\Sigma})-\cot(T_{\Sigma})\right)\\
 &&\\
 &=&\lambda_1(\Sigma)\,\dfrac{\sin \,(T_{\Sigma})\sin \,(\rho T_{\Sigma})}{\sin \,((1-\rho) T_{\Sigma})},
  \end{array}
$$
  where in the last equality we have used that $\cot(A-B)=\dfrac{\cot A\cot B+1}{\cot B-\cot A}$.

  Then \eqref{intgradu}  stands as

$$\displaystyle\int_{\Omega}|\overline{\nabla} u|^2dV\leq\dfrac{16}{3\rho T_{\Sigma}}\dfrac{c(c^3-1)}{(c^2-1)^2}+\lambda_1(\Sigma)\,\dfrac{\sin \,(T_{\Sigma})\sin \,(\rho T_{\Sigma})}{\sin \,((1-\rho) T_{\Sigma})}.$$

Making $c\to +\infty$, choosing $\rho=1/2$ and from the definition of $T_{\Sigma}$ in  \eqref{TSigma}, it follows that
  \begin{equation}\label{EG03}
 \begin{array}{rcl}
  \displaystyle\int_{\Omega}|\overline{\nabla} u|^2dV & \leq & \dfrac{32}{3T_{\Sigma}}+\lambda_1(\Sigma)\,\sin \,(T_{\Sigma})=\dfrac{32}{3\arctan(1/k_{\textrm{max}})}+\lambda_1(\Sigma)\,\sin \,(\arctan(1/k_{\textrm{max}})) \\
   &  &  \\
   & = & \displaystyle{\dfrac{32}{3\arctan(1/k_{\textrm{max}})}+\frac{\lambda_1(\Sigma)}{\sqrt{1+k_{\textrm{max}}^2}}}.
\end{array}
  \end{equation}

%
%
%
\end{proof}


\begin{lema}\label{Der}
  For all $0\leq t<T_{\Sigma}$ and $f\in C^1(\overline{\Omega})$ we have
$$\dfrac{d}{dt}\left(\int_{\Sigma_t}f(y)dS_t\right)=\int_{\Sigma_t}\langle \nabla f(y),\nabla d(y)\rangle dS_t-\int_{\Sigma_t}f(y)H_{\Sigma_t}dS_t,$$
where $d$ is the signed distance to $\Sigma$ in $\mathbb{S}^{m+1}$,  i.e.

$$
  d(y)=\left\{
                                                                          \begin{array}{ll}
                                                                            \textrm{dist }\,(y,\Sigma) &\textrm{ if}\, x\in \overline{\Omega} \\
                                                                           -\textrm{dist}\,(y,\Sigma) &\textrm{ if}\, x\in \overline{\Omega}^{c},
                                                                          \end{array}
                                                                        \right.
$$
  and $H_{\Sigma_t}$ the mean curvature of the hypersurface $\Sigma_t$.
\end{lema}
\begin{proof}
Making the change of variable $y=\Phi_t(x)$,  
  $$\begin{array}{rcl}
   \dfrac{d}{dt}\left(\displaystyle\int_{\Sigma_t}f(y)dS_t\right)  & = & \dfrac{d}{dt}\left(\displaystyle\int_{\Sigma}f(\Phi_t(x))\theta(t,x)dS\right) \\
     &  &  \\
     & = & \displaystyle\int_{\Sigma}\langle \overline{\nabla} f(\Phi_t(x)),\overline{\nabla} d(\Phi_t(x))\rangle\theta(t,x)dS+\displaystyle\int_{\Sigma}f(\Phi_t(x))\theta'(t,x)dS\\
&&\\
 & = &  \displaystyle\int_{\Sigma }\langle \overline{\nabla} f(\Phi_t(x)),\overline{\nabla} d(\Phi_t(x))\rangle dS-\displaystyle\int_{\Sigma}f(\Phi_t(x))H_{\Sigma_t}(\Phi_t(x))\theta(t,x)dS\\
&&\\
&=& \displaystyle\int_{\Sigma_t}\langle \overline{\nabla} f(y),\overline{\nabla} d(y)\rangle dS_t-\displaystyle\int_{\Sigma_t}f(y)H_{\Sigma_t}(y)dS_t.
\end{array}$$
Here we have used that  $H_{\Sigma_t}(\Phi_t(x))=-\frac{\theta'(t,x)}{\theta(t,x)}$ (see Lemma 10.9 in \cite{gray2003tubes}), where  $\theta'(t,x)$ denotes the derivative of $\theta(t,x)$  with respect to the first variable.

\end{proof}

The following result is a consequence of Lemma 3.5 in \cite{duncan2023improved}.

  \begin{lema}\label{Spruck}
   Let $0<\varepsilon\leq\frac{\Lambda}{2}$. Then for $t\in [0,\arctan(\frac{\varepsilon}{\Lambda^2})]$,
$$H_{\Sigma_t}\leq 2\Lambda.$$
   \end{lema}

 \begin{proof}
Let $0<\varepsilon\leq\frac{\Lambda}{2}$ and $t\in [0,\arctan(\frac{\varepsilon}{\Lambda^2})]$. From Lemma 3.5 in \cite{duncan2023improved} it follows that
$$H_{\Sigma_t}\leq \dfrac{\Lambda \varepsilon}{\Lambda-\varepsilon}\left(\dfrac{m}{\Lambda^2}+1\right).$$
On the other hand, since $\varepsilon\leq \Lambda/2$ and $m\leq \Lambda^2$ we have that
$$\frac{\Lambda\varepsilon}{\Lambda-\varepsilon}\left(\frac{m}{\Lambda^2}+1\right)\leq\dfrac{2\Lambda\varepsilon}{\Lambda-\varepsilon}\leq\frac{\Lambda^2}{\Lambda-\varepsilon}\leq\frac{2\Lambda^2}{\Lambda}=2\Lambda.$$
We conclude that $H_{\Sigma_t}\leq 2\Lambda$.
%
%
\end{proof}

\begin{teo}\label{DIP}Let $\Sigma$ be a closed embedded minimal hypersurface in the unit sphere $\mathbb{S}^{m+1}$ and let $\Lambda=\max\limits_{\Sigma}|A|$ be the norm of  its second fundamental form.   Assume that  $\Lambda>\sqrt{m}$ and $\varphi$ is a  first  eigenfunction of the laplacian satisfying $\displaystyle\int\limits_{\Sigma}\varphi^2\,dS=1$. Then the harmonic extension $u$ of $\varphi$ satisfies
$$ \displaystyle\int_{\Omega}u^2 dV>\dfrac{1}{4(12\Lambda+m+11)^2+1}\int_{\Omega}|\overline{\nabla} u|^2dV.$$
 \end{teo}
\begin{proof}
For $t\geq 0$, let $\Omega(t):=\{y\in \Omega:d(y)>t\}$ and $\eta(t):=\displaystyle\int_{\Omega(t)}u^2$. It follows that  $\eta(0)=\displaystyle\int_{\Omega}u^2dV$. Moreover, from the   coarea   formula

$$\begin{array}{rcl}
    \eta'_{+}(t) & = & \lim\limits_{\varepsilon\to 0^+}\dfrac{\eta(t+\varepsilon)-\eta(t)}{\varepsilon} \\
     &  &  \\
     & = &  \lim\limits_{\varepsilon\to 0^+}\dfrac{1}{\varepsilon}\left(\displaystyle\int_{\Omega_{t+\varepsilon}}u^2dV-\displaystyle\int_{\Omega_{t}}u^2dV\right) \\
     &  &  \\
  & = &-\lim\limits_{\varepsilon\to 0^+}\dfrac{1}{\varepsilon}\displaystyle\int_{\{t\leq d(y)\leq t+\varepsilon\}}u^2dV\\
&&\\
&=&-\lim\limits_{\varepsilon\to 0^+}\dfrac{1}{\varepsilon}\displaystyle\int_{t}^{t+\varepsilon}\int_{\Sigma_s}u^2dS_sds\\
&&\\
&=&-\displaystyle\int_{\Sigma_t}u^2dS_t,
  \end{array}$$

(analogously $\eta'_{-}(t)=-\displaystyle\int_{\Sigma_t}u^2dS_t$)  and so $\eta'(0)=-\displaystyle\int_{\Sigma}\varphi^2=-1$.

From Lemma \ref{Der} we have

$$\begin{array}{rcl}
   \eta''(t)  & = & -\dfrac{d}{dt}\left(\displaystyle\int_{\Sigma_t}u^2dS_t\right) \\
     &  &  \\
     & = & -\displaystyle\int_{\Sigma_t}\langle \overline{\nabla} (u^2),\overline{\nabla} d\rangle dS_t+\displaystyle\int_{\Sigma_t}u^2H_{\Sigma_t}dS_t\\
& \leq &  -\displaystyle\int_{\Sigma_t}\langle \overline{\nabla} (u^2),\overline{\nabla} d\rangle dS_t+2\Lambda\displaystyle\int_{\Sigma_t}u^2dS_t\\
&&\\
&=& \displaystyle\int_{\Omega(t)}\overline{\triangle}(u^2)dV-2\Lambda\eta'(t),
\end{array}$$

where in the last two lines we have used Lemma \ref{Spruck} and Stokes' formula   respectively. Therefore, by Proposition \ref{EG}

  \begin{equation}\label{EDO1}
    \eta''(t)+2\Lambda\eta'(t)\leq 2\displaystyle\int_{\Omega(t)}|\overline{\nabla} u|^2\leq 2\displaystyle\int_{\Omega}|\overline{\nabla} u|^2\leq 2C_1.
  \end{equation}

 Multiplying by $e^{2\Lambda t}$ both sides of (\ref{EDO1}) and integrating from $0$ to $t$ we have
  \begin{equation}\label{EDO2}
    \eta'(t)\leq 2C_1\left(\dfrac{1-e^{-2\Lambda t}}{2\Lambda}\right)-e^{-2\Lambda t}.
  \end{equation}
  Now we can integrate from $0$ to $ T_{\varepsilon}:=\arctan(\frac{\varepsilon}{\Lambda^2}) $ in (\ref{EDO2}) and deduce that
  \begin{equation}\label{EDO3}
    \eta(T_{\varepsilon})-\eta(0)\leq\frac{1}{\Lambda}C_1\left(T_{\varepsilon}+\dfrac{e^{-2\Lambda T_{\varepsilon}}-1}{2\Lambda}\right)+\dfrac{e^{-2\Lambda T_{\varepsilon}}-1}{2\Lambda}.
  \end{equation}
  Considering that $\eta(T_{\varepsilon})>0$ in (\ref{EDO3}) it follows that
\begin{equation}\label{eta0}
   \begin{array}{rcl}
 \eta(0) &>&  \dfrac{1-e^{-2\Lambda T_{\varepsilon}}}{2\Lambda}-\dfrac{1}{\Lambda}C_1\left(T_{\varepsilon}+\dfrac{e^{-2\Lambda T_{\varepsilon}}-1}{2\Lambda}\right)\\
   &  &  \\
   & = &  \dfrac{1-e^{-\Lambda T_{\varepsilon}}}{2\Lambda T_{\varepsilon}}T_{\varepsilon}-\dfrac{C_1}{2}\left(\dfrac{2\Lambda T_{\varepsilon}+e^{-2\Lambda T_{\varepsilon}}-1}{\Lambda^2T_{\varepsilon}^2}\right)T_{\varepsilon}^2
\end{array}
\end{equation}

On the other hand, for all $x\geq 0$
    \begin{equation}\label{D1}
      \dfrac{x+e^{-x}-1}{x^2}\leq\frac{1}{2}.
    \end{equation}
Using the inequality (\ref{D1}) and the fact that $T_{\varepsilon}<\frac{\varepsilon}{\Lambda^2}$  in (\ref{eta0})  we have

\begin{equation}\label{EDO4}
\begin{array}{rcl}
  \displaystyle\int_{\Omega}u^2dV & = & \eta(0) \\
   &  &  \\
   & > & T_{\varepsilon}(1-\Lambda T_{\varepsilon})-C_1 T_{\varepsilon}^2 \\
   &  &  \\
   & = &  T_{\varepsilon}\left(1-(\Lambda+C_1 )T_{\varepsilon}\right)\\
&&\\
&> &T_{\varepsilon}\left(1-\frac{\varepsilon}{\Lambda^2}(\Lambda+C_1 )\right)\\
&&\\
&=&T_{\varepsilon}\left(1-\varepsilon\left(\frac{1}{\Lambda}+\frac{C_1 }{\Lambda^2}\right)\right).
\end{array}
\end{equation}

From (\ref{EDO4}), it follows that for any $0<\varepsilon\leq \dfrac{1}{2}\left(\dfrac{1}{\Lambda}+\dfrac{C_1 }{\Lambda^2}\right)^{-1}$,

%

%

  $$\displaystyle\int_{\Omega}u^2dV > \dfrac{T_{\varepsilon}}{2}=\dfrac{1}{2}\arctan(\dfrac{\varepsilon}{\Lambda^2}).$$
In particular, for $\varepsilon=\dfrac{1}{2}\left(\dfrac{1}{\Lambda}+\dfrac{C_1 }{\Lambda^2}\right)^{-1}$ we have

$$\displaystyle\int_{\Omega}u^2dV>\dfrac{1}{2}\arctan\left(\dfrac{1}{2(\Lambda+C_1 )}\right).$$

Finally, this last inequality joint with Proposition \ref{EG} lead us to

 \begin{equation}\label{Inv} \dfrac{\displaystyle\int_{\Omega}u^2dV}{\displaystyle\int_{\Omega}|\overline{\nabla} u|^2dV}>
\dfrac{\arctan\left(\dfrac{1}{2(\Lambda+C_1 )}\right)}{2C_1 }.
\end{equation}
 
On the other hand, note that  $1\leq\km^2\le \frac{m-1}{m}\Lambda^2$.  Then, using the fact that $\lambda_1(\Sigma)\le m$ and since $\arctan x\ge \dfrac{x}{1+x^2}$ for $x\in [0,1]$,   we have
\begin{equation}\label{descor}
\begin{split}
C_1&=\dfrac{32}{3\arctan(1/\km)}+\dfrac{\lambda_1(\Sigma)}{\sqrt{1+\km^2}}\\ 
&\leq\dfrac{32(1+\km^2)}{3\km}+\dfrac{\lambda_1(\Sigma)}{\sqrt{1+\km^2}}\\ 
&<\dfrac{32\km}{3}+\dfrac{m+11}{\km}\\ 
&<\dfrac{32\Lambda}{3}+\dfrac{m+11}{\km}\\
&< 11\Lambda + m+11 .
\end{split}
\end{equation}

 We now define
 \begin{equation}\label{C2}\begin{split}
  C_2=C_2(m,\Lambda):&=\dfrac{\arctan\left(\dfrac{1}{2(\Lambda+C_1 )}\right)}{2C_1}.
  \end{split}
\end{equation} From (\ref{descor})  and (\ref{C2})  we have
\begin{equation}\label{descor1}
\begin{split}
  C_2&=\dfrac{\arctan\left(\dfrac{1}{2\Lambda+2C_1}\right)}{2C_1}\\
  &\ge\dfrac{2\Lambda+2C_1}{2C_1[(2\Lambda+2C_1)^2+1]}\\
  &\ge\dfrac{1}{(2\Lambda+2C_1)^2+1}\\   
  &>\dfrac{1}{4(12\Lambda+m+11)^2+1}.
  \end{split}
\end{equation}
Combining \eqref{Inv} and \eqref{descor1}  we have
\begin{equation*} \dfrac{\displaystyle\int_{\Omega}u^2dV}{\displaystyle\int_{\Omega}|\overline{\nabla} u|^2dV}>
\dfrac{1}{4(12\Lambda+m+11)^2+1}.
\end{equation*}
This completes the proof the theorem.

\end{proof}


We conclude the paper by 
 comparing our estimate and the lower bound for $\lambda_1(\Sigma)$ obtained by  Duncan, Sire and Spruck in \cite{duncan2023improved}. In their work, it is established that given a  closed and embedded minimal hypersurface $\Sigma$ in $\mathbb{S}^{m+1}$ with $\Lambda=\max\limits_{\Sigma}|A|\geq \sqrt{m}$, then  
\begin{equation}\label{ES}
  \lambda_1(\Sigma)\geq \dfrac{m}{2}+\dfrac{a(m)}{\Lambda^6+b(m)},
\end{equation}
 where
\begin{equation}\label{ayb}
  \begin{array}{rcl}
  a(m) & := & \displaystyle{\frac{3\sqrt{m}(m-1)}{3200} }  \left(m \arctan \left(\frac{1}{3 \sqrt{m}}\right)\right)^3\,\quad\textrm{and} \\
   &  &  \\
 b(m)  & := &\displaystyle{\frac{5(m-1)}{ 8 \sqrt{m}} } \left(m \arctan \left(\frac{1}{3 \sqrt{m}}\right)\right)^3.
\end{array}
\end{equation}

Since $\frac{x}2\le \arctan x\le x$ when $x\in [0,1]$, we have $\frac{\sqrt{m}}{6}\leq m \arctan (\frac{1}{3 \sqrt{m}})\le \frac{\sqrt{m}}3$, then from \eqref{ayb} we deduce 

\begin{equation}\label{ayb1}
  \begin{array}{rcl}
  a(m) & \leq & \frac{(m-1)m^2}{28800} \,\quad\textrm{and }\\
   &  &  \\
 b(m)  & \geq &\frac{5 (m-1)m }{1728}.
\end{array}
\end{equation}
Then, since $m\geq2$ and $\Lambda \geq \sqrt{m}$ we trivially obtain from \eqref{ayb1} that

\[\dfrac{a(m)}{\Lambda^6+b(m)}< \dfrac{ (m-1)m^2  }{28800\Lambda^6+\frac{28800}{175}}<\dfrac{ (m+1)m\Lambda^2  }{28800\Lambda^6+164}< \frac{(m+1)m}{32(12\Lambda+\Lambda^2+11)^2+8}<\frac{(m+1)m}{32(12\Lambda+m+11)^2+8}.
\]

%
%
\end{document}